\documentclass[11pt]{article}

\usepackage[utf8]{inputenc}
\usepackage[british]{babel}
\usepackage[a4paper, margin=2.3cm]{geometry}
\usepackage{amsmath, amsthm, amssymb, amsfonts}
\usepackage{mathtools}
\usepackage{thmtools}
\usepackage[shortlabels]{enumitem}
\usepackage[font={small, it}]{caption}
\usepackage{microtype}
\usepackage{xcolor}
\usepackage{mathabx}
\usepackage{hyperref}
\usepackage{bm}
\usepackage[normalem]{ulem}
\usepackage{tikz}

\newcommand{\cA}{\ensuremath{\mathcal A}}
\newcommand{\cB}{\ensuremath{\mathcal B}}

\newcommand{\cE}{\ensuremath{\mathcal E}}

\newcommand{\cH}{\ensuremath{\mathcal H}}
\newcommand{\cI}{\ensuremath{\mathcal I}}

\newcommand{\cK}{\ensuremath{\mathcal K}}

\newcommand{\cP}{\ensuremath{\mathcal P}}

\newcommand{\cS}{\ensuremath{\mathcal S}}
\newcommand{\cT}{\ensuremath{\mathcal T}}

\newcommand{\eps}{\varepsilon}
\renewcommand{\phi}{\varphi}
\renewcommand{\rho}{\varrho}

\let\setminus=\smallsetminus
\let\emptyset=\varnothing

\newcommand{\Gknp}{\mathcal{G}^{(k)}(n,p)}

\newcommand\redsout{\bgroup\markoverwith{\textcolor{red}{\rule[0.5ex]{2pt}{0.5pt}}}\ULon}
\declaretheorem[parent=section]{theorem}
\declaretheorem[sibling=theorem]{lemma}

\declaretheorem[sibling=theorem]{claim}

\declaretheorem[sibling=theorem,style=definition]{definition}

\setlist{itemsep=0.1em, topsep=0.1em, parsep=0.1em, partopsep=0.1em}

\hypersetup{
  colorlinks,
  linkcolor={red!60!black},
  citecolor={green!50!black},
  urlcolor={blue!80!black}
}

\colorlet{RoyalRed}{red!70!black}
\definecolor{RoyalBlue}{rgb}{0.25, 0.41, 0.88}
\definecolor{RoyalAzure}{rgb}{0.0, 0.22, 0.66}

\newlength{\bibitemsep}
\newlength{\bibparskip}
\let\oldthebibliography\thebibliography
\renewcommand\thebibliography[1]{%
  \oldthebibliography{#1}%
  \setlength{\parskip}{\bibitemsep}%
  \setlength{\itemsep}{\bibparskip}%
}

\title{Probabilistic hypergraph containers}
\author{
  Rajko Nenadov\thanks{School of Computer Science, University of Auckland, New Zealand. Email: \texttt{rajkon@gmail.com}.}
}
\date{}

\begin{document}
\maketitle

\begin{abstract}
Given a $k$-uniform hypergraph $\cH$ and sufficiently large $m \gg m_0(\cH)$, we show that an $m$-element set $I \subseteq V(\cH)$, chosen uniformly at random, with probability $1 - e^{-\omega(m)}$ is either not independent or is contained in an almost-independent set in $\cH$ which, crucially, can be constructed from carefully chosen $o(m)$ vertices of $I$. As a corollary, this implies that if the largest almost-independent set in $\cH$ is of size $o(v(\cH))$ then $I$ itself is an independent set with probability $e^{-\omega(m)}$. More generally, $I$ is very likely to inherit structural properties of almost-independent sets in $\cH$.

The value $m_0(\cH)$ coincides with that for which Janson's inequality gives that $I$ is independent with probability at most $e^{-\Theta(m_0)}$. On the one hand, our result is a significant strengthening of Janson's inequality in the range $m \gg m_0$. On the other hand, it can be seen as a probabilistic variant of hypergraph container theorems, developed by Balogh, Morris and Samotij and, independently, by Saxton and Thomason. While being strictly weaker than the original container theorems in the sense that it does not apply to all independent sets of size $m$, it is nonetheless sufficient for many applications and admits a short proof using probabilistic ideas.
\end{abstract}

\section{Introduction}

Let $\cH$ be a $k$-uniform hypergraph, a \emph{$k$-graph} for short, for some $k \in \mathbb{N}$. Throughout the paper we use $N$ (and sometimes $v(\cH)$) to denote the number of vertices in $\cH$, and $e(\cH)$ to denote the number of hyperedges. What is the probability that an $m$-element subset $I \subseteq V(\cH)$, chosen uniformly at random among all $m$-element subsets, is an independent set in $\cH$? This question is addressed by Janson's inequality:
\begin{equation} \label{eq:janson}
    \Pr[I \text{ is independent}] < C e^{- \mu_{\cH}(m)^2 / \Delta_{\cH}(m)}
\end{equation}
for some sufficiently large (absolute) constant $C > 1$, where
$$
    \mu_\cH(m) = e(\cH) (m / N)^k
$$
roughly corresponds to the expected number of edges induced by $I$, and
\begin{gather*}
    \Delta_{\cH}(m) = \sum_{(e, e') \in \Lambda_\cH} (m/N)^{|e \cup e'|} \\
    \Lambda_\cH = \left\{ (e, e') \in \cH \times \cH \colon e \cap e' \neq \emptyset \right\}
\end{gather*}
corresponds to the usual estimate of the variance of this number. The standard version of Janson's inequality is stated for \emph{binomial} random subsets, that is, when $I$ is formed by taking each element with  probability $p = m / N$, independently of all other elements (e.g.\ see \cite{alon_spencer}). Inequality \eqref{eq:janson} follows from it by standard concentration bounds, and we refer the reader to \cite[Lemma 5.2]{alon14sumfree} for details.

In this paper, we are interested in the case where \eqref{eq:janson} gives that $I$ is independent with probability at most $e^{-\Theta(m)}$, which happens for $\Delta_\cH(m) = O(\mu_\cH(m)^2 / m)$. When $I$ is a binomial random subset with $p = m / N$, this is the correct order of magnitude as $I$ is an empty set (hence independent) with probability $(1-p)^N \approx e^{-m}$. The inequality is also correct in many instances with respect to uniform sampling, however, it is perhaps less known (and somewhat surprising) that there are cases where \eqref{eq:janson} significantly overestimates the true probability, that is, where the true probability of $I$ being independent decays as $e^{-\omega(m)}$. Let us look at two examples of $3$-graphs, $\cT_n$ and $\cA_n$, which demonstrate this. The vertices of $\cT_n$ correspond to the edges of $K_n$, a complete graph with $n$ vertices (hence $N = \binom{n}{2}$), and three vertices form a hyperedge if the corresponding edges in $K_n$ form a triangle. The vertices of $\cA_n$ are integers $\{1, \ldots, n\}$ (hence $N = n$), and three vertices (that is, numbers) form a hyperedge if they form a $3$-term arithmetic progression. It is an easy exercise to show that  $\Delta_{\cH}(m) = \Theta(\mu_{\cH}(m)^2 / m)$ if $m \ge N^{3/4}$ ($\cH = \cT_n)$ and $m \ge N^{1/2}$ ($\cH = \cA_n$). In the case where $\cH = \cT_n$, the probability of $I$ being independent is indeed $e^{-\Theta(m)}$ in this range, however for $\cH = \cA_n$ it decays as $e^{-\omega(m)}$ for $m \gg \sqrt{N}$ (see \cite{nenadov2021small, samotij15counting} for self-contained proofs, and \cite{balogh2017kap} for further refinements). Note that this does not contradict the discussion about the binomial case as the event of sampling an empty set here happens with probability $0$ (i.e.\ we always have exactly $m$ elements).

The previous two examples show that inequality \eqref{eq:janson}, in general, cannot be improved, but also that parameters $\mu_\cH$ and $\Delta_\cH$ do not capture all the relevant aspects of $\cH$. In particular, the main qualitative difference between $\cT_n$ and $\cA_n$, not captured by these parameters, is the size of a largest independent set. A largest independent set in $\cT_n$ is of size roughly $N/2$ (Mantel's theorem), and in $\cA_n$ of size $o(N)$ (Roth's theorem), and we can lower bound the probability of $I$ being independent by the probability it is a subset of such a (fixed) independent set, which is $e^{-\Theta(m)}$ and $e^{-\omega(m)}$ respectively. Our main result, or rather its corollary, shows that the size of a largest (almost-)independent set is, indeed, a missing component in \eqref{eq:janson}. Briefly, it states that with overwhelmingly high probability a sampled set $I$ is either not independent or is a subset of an almost independent set in $\cH$. This would be rather trivial -- after all, if $I$ is an independent set then $I \subseteq I$ makes the previous statement vacuously hold -- if it was not for the additional fact that such an almost-independent set can be constructed by looking only at some carefully chosen $o(m)$ vertices of $I$. As we will shortly see, this implies that the previously discussed lower bound coming from the probability that we sample a subset of a fixed independent set in $\cH$ gives roughly the correct exponent.

To state the result concisely we need a few more definitions. Given a subset $V' \subseteq V(\cH)$, we use $e(V')$ as shorthand for $e(\cH[V'])$, the number of hyperedges in the subgraph of $\cH$ induced by $V'$. Given $\eps > 0$, denote with $\cI_{\eps}(\cH)$ the family of all subsets $S \subseteq V(\cH)$ with $e(S) \le \eps e(\cH)$. Finally, given a set $F$, let $\cP(F)$ denote the family of all subsets of $F$.

\begin{theorem}[Probabilistic Hypergraph Containers] \label{thm:containers}
    For every $k \in \mathbb{N}$ and $\eps, B > 0$, there exists $\gamma, T > 0$ such that the following holds. Let $\cH$ be an $N$-vertex $k$-graph, and suppose $m_0 \in \mathbb{N}$ $(m_0 < N)$ satisfies
    \begin{equation} \label{eq:delta_m0}
        \Delta_{\cH}(m_0) \le B \mu_\cH^2(m_0) / m_0.
    \end{equation}
    Then for every $Tm_0 \le m < N$ there exists a function $f_m \colon \cP(V(\cH))^{k-1} \rightarrow \cI_\eps(\cH)$ such that an $m$-element $I \subseteq V(\cH)$, chosen uniformly at random, with probability at least $1 - \eps^m$ satisfies (at least) one of the following properties:
    \begin{enumerate}[(P1)]
        \item \label{prop:i} $e(I) \ge \gamma \cdot \mu_\cH(m)$, or
        \item \label{prop:ii} There exists $F \subseteq I$, $|F| = \eps m$, such that for some $\mathbf{F} \in \cP(F)^{k-1}$ we have
        $$
            I \subseteq F \cup f_m(\mathbf{F}).
        $$
    \end{enumerate}
\end{theorem}

Let us describe a typical application of Theorem \ref{thm:containers}. Suppose $\cH$ is such that $e(S) < \eps e(\cH)$ implies $|S| \le \beta N$, where $\beta \rightarrow 0$ as $\eps \rightarrow 0$ and $N \rightarrow \infty$ (technically, we need to consider a family of hypergraphs $\{\cH_i\}$). This is, for example, the case for $\cH = \cA_n$ due to Roth's theorem, and for hypergraphs constructed from longer arithmetic progressions due to Szemer\'edi's theorem. We use Theorem \ref{thm:containers} to estimate the probability that an $m$-element $I \subseteq V(\cH)$ sampled uniformly at random, for sufficiently large $m$, is independent. Choose $\eps > 0$, and suppose $m_0$ satisfies \eqref{eq:delta_m0} and $m \ge T(\eps) m_0$. For each $F \subseteq V(\cH)$ of size $\eps m$ and $\mathbf{F} \in \cP(F)^{k-1}$, consider the set $A = f_m(\mathbf{F})$. This gives us a family of $t = \binom{N}{\eps m} 2^{\eps m k}$ pairs of sets, $(F_1, A_1), \ldots, (F_t, A_t)$, where each $F_i$ is of size $\eps m$ and $A_i$ induces at most $\eps e(\cH)$ edges and, therefore, is of size $|A_i| \le \beta N$. Let $\cE_I$ denote the event that $I$ is an independent set, and let $\cE_T$ denote the event that $I$ satisfies the conclusion of Theorem \ref{thm:containers}. Then
$$
    \Pr[\cE_I] = \Pr[\cE_I \mid \cE_T] \Pr[\cE_T] + \Pr[\cE_I \mid \overline{\cE_T}] \Pr[\overline{\cE_T}] 
    \le \Pr[\cE_I \mid \cE_T] + \Pr[\overline{\cE_T}] \le \Pr[\cE_I \mid \cE_T] + \eps^m.
$$
Conditioning on the event $\cE_T$, if $I$ is an independent set then $F_i \subseteq I$ and $I \setminus F_i \subseteq A_i$ for some $i \in [t]$. Calculating the probability that this happens involves manipulation of binomial coefficients, and it eventually gives $e^{-\zeta m}$ for $\zeta \rightarrow 0$ as $\beta \rightarrow 0$ (see the proof of \cite[Theorem 1.1]{balogh2015indset} for details). Therefore, $I$ is independent with probability $e^{-\omega(m)}$ as $m \gg m_0$. As a comparison, Janson's inequality implies that $I$ is independent with probability at most $e^{-\Theta(m)}$ in this range of $m$.

A reader familiar with recent developments in extremal and probabilistic combinatorics will notice that Theorem \ref{thm:containers} is, in fact, an approximate (or probabilistic) version of the result of Balogh, Morris, and Samotij \cite{balogh2015indset} and Saxton and Thomason \cite{saxton2015containers}, colloquially called \emph{hypergraph container theorems} (see also an excellent survey by Balogh, Morris, and Samotij \cite{balogh2018survey}). In particular, both of these results state that the described property holds with probability exactly $1$, that is, for every $I$ which does not satisfy \ref{prop:i} the part \ref{prop:ii} holds. With very little effort, hypergraph containers imply almost all known extremal results in random graphs, some of which were originally proven in breakthroughs by Schacht \cite{schacht2016extremal} and Conlon and Gowers \cite{conlon2016combinatorial}, and many counting results, such as the celebrated K\L R conjecture \cite{kohayakawa1997k} or the number of maximal triangle-free graphs \cite{balogh14maximal}.  All of these results can also be derived from Theorem \ref{thm:containers} in much the same way, and we demonstrate its use to prove one such new result in Section \ref{sec:application}. That being said, there are also results, such as \cite{balogh18solymosi,ferber20super,morris2016cyclefree}, where Theorem \ref{thm:containers} does not suffice. These examples rely on an iterative application of the containers, for which Theorem \ref{thm:containers}, due to the existence of an exceptional family of `bad' $m$-element subsets, is not suited.

As remarked earlier, Theorem \ref{thm:containers} strengthens Janson's inequality in the higher range of $m$. However, it does not imply Janson's inequality in the lower range, thus the two are not comparable. The upper bound on $\Delta_\cH(m_0)$ is very close in spirit to the notion of \emph{$(K, p)$-boundedness} from \cite{schacht2016extremal} and the assumption in \cite[Theorem 2.1]{balogh2020efficient}. The main advantage of working directly with $\Delta_\cH(m_0)$ is that it is exactly the parameter used in Janson's inequality, making the cases where a container-type statement applies more transparent. The main value of our new proof lies in the simplicity and transparency of the ideas which, in our view, exploit the very essence of why the existence of the containers (i.e.\ the existence of $F$ and $f$ satisfying \ref{prop:ii}) is not surprising. 

\subsection{Proof outline}

The proofs of container theorems from \cite{balogh2015indset} and \cite{saxton2015containers} are roughly along the same lines and differ mainly in the analysis of an otherwise very similar algorithm for finding a subset $F \subseteq I$ and constructing containers. A few other proofs have been obtained since, including the recent work by Balogh and Samotij \cite{balogh2020efficient} which provides almost optimal dependency of parameters, a short proof by Bernshteyn, Delcourt, Towsner, and Tserunyan \cite{bernshteyn2019short}, and another simpler proof by Saxton and Thomason \cite{saxton2016simple} of a variant which only applies to linear hypergraphs. The proof we present here uses probabilistic ideas and, at its core, relies on the \emph{deletion method} of R\"odl and Ruci\'nski.

In the remainder of this section we discuss two things: which property the bound \eqref{eq:delta_m0} implies, and how does such a property give the existence of the function $f_m$ and a suitable subset $F$ in \ref{prop:ii}? 

The intuition behind the assumption \eqref{eq:delta_m0} is best described on a simple example $\cH = \cT_n$. We naturally refer to the vertices of $\cT_n$ as edges, and hyperedges in $\cT_n$ as triangles. Recall that $N = \binom{n}{2}$ and $e(\cH) = \binom{n}{3}$. As mentioned earlier, $\Delta_\cH(m) = \Theta(\mu_\cH(m)^2 / m)$ holds for $m \ge N^{3/4}$, that is $m \ge n^{3/2}$. What is important about this value of $m$ is that every edge $e \in K_n$, with constant positive probability, forms a triangle with two edges from a randomly chosen $m$-element $I \subseteq E(K_n)$. Indeed, each edge $e$ belongs to $n-2$ different triangles, and the probability that in each such triangle we sample at most one other edge is at most 
$$
(1 - (m/N)^2)^{n-2},
$$
which is constant for $m \ge n^{3/2}$. One can apply a similar argument to conclude that in any $S \subseteq E(K_n)$ which spans at least $\eps n^3$ triangles, there are $\Theta_\eps(n^2)$ edges in $S$ which, with constant positive probability, form a triangle with two edges from a randomly chosen $m$-element $I' \subseteq S$. In other words, $\Theta_\eps(n^2)$ edges in $S$ form a triangle with some two edges from $I'$, in expectation. 

Let us now see how to use the described property. Instead of looking at the whole $I$ at once, we `reveal' it in pieces $I = I_1 \cup I_2 \cup \ldots$. From the previous discussion, we expect $\Theta(n^2)$ edges in $E(K_n)$ to form a triangle with some two edges from $I_1$, for a randomly chosen $I_1$ of size $\xi m > m_0$. Let us denote these edges with $\hat L_1$, and note that if $I$ is to be independent (that is, triangle-free) then $I \setminus I_1 \subseteq L_1$, where $L_1 = E(K_n) \setminus \hat L_1$. If $L_1$ spans less than $\eps n^3$ triangles, then we could set $f_m(I_1) = L_1$ and we are done. Otherwise, $L_1$ spans enough triangles, thus, again, we expect $\Theta_\eps(n^2)$ edges in $L_1$ to form a triangle with two edges from a randomly chosen $I_2 \subseteq L_1$ of size $|I_2| = \xi m$. Defining $L_2$ to be the set of edges in $L_1$ which do not form a triangle with two edges from $I_2$, we further have $I \setminus (I_1 \cup I_2) \subseteq L_2$, and so on. By repeating this argument constantly many rounds, we eventually arrive at a subset of edges $L_z$, for some constant $z$, which contains less than $\eps n^3$ triangles. Crucially, by taking $\xi$ to be sufficiently small (and thus $m \ge T N^{3/4}$ for $T$ sufficiently large), we can keep the total size of the revealed part of $I$ to be smaller than $\eps m$.

The actual proof proceeds in a somewhat different manner, but the described argument is implicitly present. The assumption \eqref{eq:delta_m0} is also used more directly, through Paley--Zygmund inequality, to deduce that there are many vertices in a $k$-graph $\cH$ which form a hyperedge with some $k-1$ vertices from a certain small subset. Finally, at this point we are not in a position to say more about where the probability $\eps^m$ comes from, other than it is obtained through an application of the \emph{deletion lemma} of R\"odl and R\'ucinski \cite{rodl1995threshold}.

\section{Proof of the probabilistic containers}

Given subsets $D, W \subseteq V(\cH)$ and $k' \in [k-1]$, let $e_{k'}(D, W)$ denote the number of edges in $\cH[D \cup W]$ which intersect $W$ in at least $k'$ vertices. Similarly, for a vertex $v \in D$ let $\deg_{k'}(v, D, W)$ denote the number of edges in $\cH[D \cup W]$ which contain $v$ and intersect $W \setminus \{v\}$ in at least $k'$ vertices.

\begin{definition} \label{def:saturating}
    Let $\cH$ be a $k$-graph and $k' \in [k-1]$. A subset $W \subseteq V(\cH)$ is \emph{$(k', \alpha, t)$-saturating} for $D \subseteq V(\cH)$, for some $\alpha, t > 0$, if
    \begin{equation} \label{eq:saturated_def}
        \left| \left\{ v \in D \colon \deg_{k'}(v, D, W) \ge t / N \right\} \right| \ge \alpha N.
    \end{equation}
\end{definition}

The following lemma is the main building block in the proof of Theorem \ref{thm:containers}. Note, it is crucial that $\alpha$ does not depend on $\xi$.

\begin{lemma} \label{lemma:saturating}
    For every $k \in \mathbb{N}$ and $\xi, \gamma, \beta, B > 0$, there exists $\alpha = \alpha(k, \beta, B), \lambda, T > 0$ such that the following holds. Let $\cH$ be a $k$-graph and suppose $m_0 \in \mathbb{N}$ $(m_0 < N)$ satisfies
    $$
        \Delta_{\cH}(m_0) \le B \mu_{\cH}^2(m_0)/m_0.
    $$
    Let $T m_0 \le m \le N$. Then there exists a family $\cB_m$ of $m$-element subsets of $V(\cH)$, $|\cB_m| < \beta^m \binom{N}{m}$, such that every $m$-element $I \subseteq V(\cH)$, $I \not \in \cB_m$, has the following property: There exists $X \subseteq I$, $|X| \le \beta m$, such that if some $I' \subseteq I \setminus X$,  $|I'| \ge m/2$, and $D \subseteq V(\cH)$ satisfy
    \begin{equation} \label{eq:many_edges}
        e_{k'}(D, I') \ge \gamma \cdot e(\cH) q^{k'},
    \end{equation}
    where $q = m / N$ and $k' \in [k-1]$, then $I'$ contains a $(k', \alpha, \lambda e(\cH) q^{k'})$-saturating set $W$ for $D$ of size $|W| \le \xi m$.
\end{lemma}

Let us give a brief intuition behind the statement of Lemma \ref{lemma:saturating}. First, if $D = V(\cH)$ and $I'$ is chosen uniformly at random, then $e(\cH) q^{k'}$ roughly corresponds to the expected number of edges which intersect $I'$ in at least $k'$ elements. If \eqref{eq:many_edges} holds and the edges intersecting $I'$ are evenly distributed, then we expect many vertices $v \in D$ to satisfy the degree condition in \eqref{eq:saturated_def} with $t = \Omega(e(\cH) q^{k'})$ and $I'$ having the role of $W$. By choosing $W \subseteq I'$ uniformly at random we can hope that it satisfies a scaled down version of \eqref{eq:many_edges} and inherits the distribution of the edges intersecting it. Consequently, $W$ is saturating for $D$. The key part of the lemma is that properties which are sufficient to guarantee such a distribution of edges hold for all but at most $\beta^m$-fraction of $m$-element subsets. The proof is based on the R\"odl--Ruci\'nski deletion method, and we postpone it for the next section. Instead, we now prove Theorem \ref{thm:containers}.

\begin{proof}[Proof of Theorem \ref{thm:containers}]
    Let us start by fixing the constants. Set $\beta = \eps / (2k^2)$, $\lambda_0 = \eps$, and for each $k' = 1, \ldots, k-1$, iteratively, set
    $$
        \gamma_{k'} = \lambda_{k'-1} \eps / (2k^2), \;
        \alpha_{k'} = \alpha_{\ref{lemma:saturating}}(k, \beta, B), \;
        \xi_{k'} = \alpha_{k'} \eps / (4k^2), \;
        \lambda_{k'} = \lambda_{\ref{lemma:saturating}}(k, \xi_{k'}, \gamma_{k'}, \beta, B). \;
    $$
    Note that it is crucial here that $\alpha$ in Lemma \ref{lemma:saturating} does not depend on $\xi$, as otherwise we would get a circular dependency. Finally, set $T = \max_{k' \in [k-1]} T_{\ref{lemma:saturating}}(k, \xi_{k'}, \gamma_{k'}, \beta, B)$. We prove the statement for $\gamma = \gamma_k = \lambda_{k-1} \eps / (2k^2)$. Throughout the proof we use $q = m / N$.
    
    For each $k' \in [k-1]$, let $\cB_m^{k'}$ be the family of $m$-element subsets given by Lemma \ref{lemma:saturating} for $\xi_{k'}$ (as $\xi$), $\gamma_{k'}$ (as $\gamma$), $\beta$ and $B$. Take $\cB_m$ to be the collection of all the sets from these families, thus $|\cB_m| < \eps^m \binom{N}{m}$. We show that every $m$-element $I \not \in \cB_m$ such that $e(I) < \gamma \cdot \mu_\cH(m)$ satisfies \ref{prop:ii}.
    
    \paragraph{Function $f_m$.} Given $\mathbf{F} = (F_1, \ldots, F_{k-1}) \in \cP(V(\cH))^{k-1}$, set $D_k := V(\cH)$ and iteratively define $D_{k'}$ for $k' = k-1, \ldots, 0$ as follows:
    \begin{equation} \label{eq:D_k}
        D_{k'} = \left\{v \in D_{k' + 1} \colon \deg_{k'}(v, D_{k' + 1}, F_{k'}) <  t_{k'} / N  \right\},
    \end{equation}
    where
    $$
        t_{k'} = \lambda_{k'} \cdot e(\cH) q^{k'}.
    $$
    In the case $k'=0$ we slightly abuse the notation by using  non-defined $F_0$, which is actually irrelevant for the definition of $D_0$. Finally, we set $f_m(\mathbf{F}) = D_0$. Note that $e(D_0) \le \lambda_0 e(\cH)$ by the definition of $t_0$, and furthermore $\lambda_0 e(\cH) = \eps e(\cH)$ by the choice of $\lambda_0$. Therefore, we have $f_m(\mathbf{F}) \in \cI_\eps(\cH)$, as required.
     
    \paragraph{Finding $\mathbf{F}$.} Consider some $m$-element $I \not \in \cB_m$, and suppose \ref{prop:i} does not hold. Let $X \subseteq I$ be the union of the sets $X^{k'}$ promised by Lemma \ref{lemma:saturating} for $k' \in [k-1]$, with parameters associated to each $k'$ as stated above. Then $|X| < \eps m/ (2k)$ by the choice of $\beta$. Set $I' = I \setminus X$. We find $F_1, \ldots, F_{k-1} \subseteq I'$ using the following algorithm:
    \begin{itemize}
        \item Set $D_k = V(\cH)$.
        \item For $k' = k-1, \ldots, 1$:
        \begin{enumerate}
            \item Initially set $F_{k'} = \emptyset$ and $D_{k'} = D_{k'+1}$.
            \item While $e_{k'}(D_{k'}, I') \ge \gamma_{k'} \cdot e(\cH) q^{k'}$: 
            \begin{enumerate}[(a)]
                \item Let $W \subseteq I'$ be a $(k', \alpha_{k'}, t_{k'})$-saturating set for $D_{k'}$ of size $|W| \le \xi_{k'} m$ \\ (guaranteed to exists by Lemma \ref{lemma:saturating}). \label{alg:W}
                \item Update $F_{k'} = F_{k'} \cup W$.
                \item Set $D_{k'}$ to be as defined in \eqref{eq:D_k} with respect to the new set $F_{k'}$.
            \end{enumerate}
        \end{enumerate}
    \end{itemize}
    
    % Next, we show there exist subsets $F_{k-1}, \ldots, F_1 \subseteq I'$, each of size $|F_i| \le \eps m / (2k)$, such that
    % \begin{equation} \label{eq:D_k_upper}
    %     e_{k'}(D_{k'}, I') < \gamma_{k'} e(\cH) q^{k'}
    % \end{equation}
    % for every $k' \in \{k-1, \ldots, 1\}$, where $D_{k'}$ is as defined in \eqref{eq:D_k} with respect to these $F_i$'s. As $I$ does not satisfy \ref{prop:i} we also have $e_k(D_k, I') = e(I') < \gamma_k e(\cH) q^k$.
    
    % We find such sets iteratively for $k' = k-1, \ldots, 1$: Set $F_{k'} = \emptyset$ and $D_{k'} = D_{k'+1}$, where $D_{k'+1} = V(\cH)$ if $k'+1=k$ and otherwise it is given by \eqref{eq:D_k} with respect to already defined $F_{k-1}, \ldots, F_{k'+1}$. As long as \eqref{eq:D_k_upper} is not satisfied pick a $(k', \alpha_{k'},t_{k'})$-saturating set $W \subseteq I'$ for $D_{k'}$ of size $|W| \le \xi_{k'} m$, and update $F_{k'} = F_{k'} \cup W$ and $D_{k'}$, again as defined in \eqref{eq:D_k} (with respect to the new $F_{k'}$). 
    
    Consider one particular $k' \in \{k-1, \ldots, 1\}$. By the choice of $W$ in \ref{alg:W}, in each iteration of the while loop the set $D_{k'}$ decreases by at least $\alpha_{k'} N$. Therefore, after at most $\lceil 1 / \alpha_{k'} \rceil$ iterations we obtain a set $F_{k'}$ of size at most $\eps m / (2k)$ (owing to the choice of $\xi_{k'}$) such that the set $D_{k'}$ satisfies 
    \begin{equation} \label{eq:D_k_upper}
         e_{k'}(D_{k'}, I') < \gamma_{k'} \cdot e(\cH) q^{k'}.
    \end{equation}
    As $I$ does not satisfy \ref{prop:i} we also have $e_k(D_k, I') = e(I') < \gamma_k \cdot e(\cH) q^k$.    
    
    It remains to show that $D_0$, which is equal to $f_m(\mathbf{F})$ for $\mathbf{F} = (F_1, \ldots, F_{k-1})$, contains almost all the elements from $I'$. To this end, let $R \subseteq I'$ denote the set of all elements which are not contained in $D_0$. Suppose, towards a contradiction, that $|R| \ge \eps m / 2$. Then for some $k' \in \{0, \ldots, k-1\}$ we have $|R_{k'}| \ge \eps m / (2k)$, where $R_{k'} = R \cap (D_{k'+1} \setminus D_{k'})$. This implies
    \begin{align*}
        e_{k'+1}(D_{k'+1}, I') &\ge \frac{1}{k} \sum_{v \in R_{k'}} \deg_{k'}(v, D_{k'+1}, I') \ge \frac{1}{k} \sum_{v \in R_{k'}} \deg_{k'}(v, D_{k'+1}, F_{k'}) \\
        &\ge \frac{\eps m}{2k^2} \cdot \frac{t_{k'}}{N} = \gamma_{k'+1} e(\cH) q^{k'+1},
    \end{align*}
    which contradicts \eqref{eq:D_k_upper}. Therefore, $|R| < \eps m  / 2$. All together, the set $F = X \cup R \cup F_1 \cup \ldots \cup F_{k-1}$ is of size $|F| \le \eps m$ and $I \subseteq F \cup f_m(\mathbf{F})$. By adding arbitrary elements to $F$ if needed, we can assume $|F| = \eps m$.
\end{proof}

\section{Saturating sets} \label{sec:second_moment}

% The proof of Lemma \ref{lemma:saturating} revolves around the following idea. Let $I, D \subseteq V(\cH)$ be such that $e_{k'}(D, I) \ge t$, for some $t$ and $k' \in [k-1]$. To say something about the distribution of the edges counted in $e_{k'}(D, I)$, it suffices to bound the second moment of the random variable $\deg_{k'}(v, D, I)$ with respect to uniformly chosen $v \in V(\cH)$. More precisely, if
% $$
%     \mathbb{E}_v[\deg_{k'}^2(v, D, I)] \le Z(t/N)^2,
% $$
% for some $Z > 1$, then by the Paley-Zygmund inequality we have
% $$
%     \textrm{Pr}_v \left( \deg_{k'}(v, D, I) \ge t/(2N) \right) \ge \frac{\mathbb{E}_v[\deg_{k'}(v, D, I)]^2}{4\mathbb{E}_v[\deg_{k'}^2(v, D, I)]} = 1 / (4Z).
% $$
% In other words, $I$ is $(k', 1/(4Z), t/2)$-saturating for $D$.

The following lemma captures the main technical property used in the proof of Lemma \ref{lemma:saturating}.

\begin{definition}
Let $\cH$ be a $k$-graph. Given a subset $Q \subseteq V(\cH)$, $x \in \mathbb{N}_0$ and $\xi > 0$, set
$$
    \Delta_{\cH}^Q(\xi, x) = \sum_{(e, e') \in \Lambda_{\cH}} Q_x(e \cup e') \xi^{|e \cup e'| - x},
$$
where
$$
    Q_x(S) = \begin{cases}
    1, &\text{if } |S \setminus Q| \le x \\
    0, &\text{otherwise.}
    \end{cases}
$$
\end{definition}

Note that $Q_0(S)$ is the indicator function for $S \subseteq Q$.

\begin{lemma} \label{lemma:Lambda}
    For every $k \in \mathbb{N}$ and $\beta > 0$, there exists $K > 1$ such that the following holds. Let $\cH$ be a $k$-graph and $m \in \mathbb{N}$ $(m < N)$. Then an $m$-element subset $I \subseteq V(\cH)$, chosen uniformly at random, with probability at least $1 - \beta^m$ contains $X \subseteq I$, $|X| \le \beta m$, such that,   for every integer $x \in [0, 2k-1]$ and every $\xi > 0$, we have
    \begin{equation} \label{eq:Delta_xi_x}
        \Delta_{\cH}^{I \setminus X}(\xi, x) \le K \Delta_{\cH}(\xi m) \cdot \left( \xi m / N \right)^{-x}.
    \end{equation}
\end{lemma}

We postpone the proof of Lemma \ref{lemma:Lambda}, and instead derive Lemma \ref{lemma:saturating} first.

\begin{proof}[Proof of Lemma \ref{lemma:saturating}]
    Let $\cB_m$ be the family of `bad' $m$-element subsets of $V(\cH)$ implied by Lemma \ref{lemma:Lambda}. We show that an $m$-element subset $I \subseteq V(\cH)$, $I \not \in \cB_m$, satisfies the property of the lemma.
    
    Let $X \subseteq I$ be a subset of size $|X| \le \beta m$, promised by Lemma \ref{lemma:Lambda}. Let $k' \in [k-1]$, and suppose $e_{k'}(D, I') \ge \gamma e(\cH) q^{k'}$ for some $I' \subseteq I \setminus X$, $|I'| \ge m/2$, and $D \subseteq V(\cH)$ (recall $q = m / N$).
    
    \begin{claim} \label{claim:W}
    There exist a subset $W \subseteq I'$ with the following properties:
    \begin{enumerate}[(a)]
        \item $|W| \le \xi m$, \label{prop:Wa}
        \item $e_{k'}(D, W) \ge \sigma \cdot \gamma e(\cH) (q \xi)^{k'}$, and \label{prop:Wb}
        \item $\sum_{v \in V(\cH)} \deg_{k'}^2(v, W) \le N \cdot Z \left( \gamma e(\cH) (q \xi)^{k'} / N \right)^2$, \label{prop:Wc}
    \end{enumerate}
    where $\sigma, Z > 0$ are constants which do not depend on $\xi$ and $\gamma$, and $\deg_{k'}(v, W)$ denotes the number of edges in $\cH$ which contain $v$ and at least $k'$ vertices from $W \setminus \{v\}$.
    \end{claim}
    
    Suppose Claim \ref{claim:W} holds, and let $W \subseteq I'$ be a promised subset. We show that $W$ is saturating for $D$. Consider $v \in V(\cH)$ chosen uniformly at random and let $Y = \deg_{k'}(v, D, W)$ (for $v \not \in D$ we have $\deg_{k'}(v, D, W) = 0$). Then $\mathbb{E}_v[Y] \ge e_{k'}(D, W) / N$. As $\deg_{k'}(v, D, W) \le \deg_{k'}(v, W)$, from \ref{prop:Wb} and \ref{prop:Wc} we conclude 
    $$
        \mathbb{E}_v[Y^2] \le \mathbb{E}_v[Y]^2 Z / \sigma^2.
    $$
    Applying the Paley--Zygmund inequality, we get
    $$
        \textrm{Pr}_v \left( Y \ge \mathbb{E}_v[Y]/2 \right) \ge \frac{\mathbb{E}_v[Y]^2}{4 \mathbb{E}_v[Y^2]} \ge \sigma^2 /(4 Z) =: \alpha.
    $$
    In other words, for at least $\alpha N$ elements $v \in V(\cH)$ we have 
    $$
        \deg_{k'}(v, D, W) \ge \lambda e(\cH) q^{k'} / N,
    $$
    where $\lambda = \sigma \gamma \xi^{k'} / 2$, hence $W$ is $(k',\alpha,\lambda e(\cH)q^{k'})$-saturating for $D$. This finishes the proof of Lemma \ref{lemma:saturating}, pending on the proof of Claim \ref{claim:W}.
\end{proof}
    
\begin{proof}[Proof of Claim \ref{claim:W}]
    Throughout the proof, we rely on the fact that for $m \ge m_0$ we have $\Delta_\cH(m) \le (m/m_0)^{2k-1} \Delta_{\cH}(m_0)$ and, consequently, $\Delta_\cH(m) \le B \mu_{\cH}(m)^2 / m$ by the assumption of Lemma \ref{lemma:Lambda}.
    
    We prove the existence of a desired subset $W \subseteq I'$ using the probabilistic method. In particular, we show that a subset $W \subseteq I'$ formed by taking each element in $I'$ with probability $\xi/2$, independently of all other elements, satisfies \ref{prop:Wa}--\ref{prop:Wc} (simultaneously) with positive probability. For start, we have $\Pr[\ref{prop:Wa}] \ge 1/2$ by Markov's inequality. Using Paley--Zygmund inequality we next show $\Pr[\ref{prop:Wb}] \ge 4/5$, and then, again, using Markov's inequality, $\Pr[\ref{prop:Wc}] \ge 4/5$.
    
    Let $L = e_{k'}(D, W)$, and for each $e \in \cH$ let $L_e$ be an indicator variable for the event $|e \cap W| \ge k'$. Then $L = \sum_{e \in \cH[D \cup I']} X_e$, thus by the linearity of expectation
    $$
        \mathbb{E}[L] \ge e_{k'}(D, I') \cdot (\xi/2)^{k'} >\gamma e(\cH) (q \xi)^{k'} / 2^k,
    $$
    where the first inequality follows from a fact that for every edge $e$ such that $|e \cap I'| < k'$ we deterministically have $L_e = 0$. On the one hand, this implies that if $|e \cap I'| < k$ then $L_e$ is independent of all other variables. On the other hand, if some variables $L_e$ and $L_e'$ are not independent then we necessarily have $|e \cap I'| \ge k$ and $|e' \cap I'| \ge k$, thus
    \begin{equation} \label{eq:e2kk}
        |(e \cup e') \setminus I'| \le |e \setminus I'| + |e' \setminus I'| \le 2(k - k').
    \end{equation}
    We now estimate $\mathbb{E}[L^2]$ as follows:
    \begin{align*}
        \mathbb{E}[L^2] &\le \mathbb{E}[L]^2 + \sum_{e \sim e'} \mathbb{E}[L_e L_{e'}] \le \mathbb{E}[L]^2 + \sum_{e \sim e'} 1 \\ 
        &\stackrel{\eqref{eq:e2kk}}{\le} \mathbb{E}[L]^2 + \sum_{(e, e') \in \Lambda_\cH} Q_{2k-2k'}(e \cup e') = \mathbb{E}[L]^2 + \Delta_{\cH}^{I\setminus X}(1, 2k - 2k'),
    \end{align*}
    where the first two sums go over (ordered) pairs of edges such that $L_e$ and $L_{e'}$ are not independent. As $I \setminus X$ satisfies the property of Lemma \ref{lemma:Lambda}, for $K = K_{\ref{lemma:Lambda}}(k, \beta)$ we have
    $$
        \Delta_{\cH}^{I \setminus X}(1, 2k-2k') \stackrel{\eqref{eq:Delta_xi_x}}{\le} K \Delta_{\cH}(m) \cdot q^{-2k + 2k'} \le K ( B \mu_\cH(m)^2 / m) \cdot q^{-2k + 2k'} \le \frac{2^{2k} K B}{\xi^{2k} \gamma^2 m} \cdot \mathbb{E}[L]^2.
    $$
    Finally, for $\sigma = 1/(2^k \cdot 10)$ in \ref{prop:Wb}, Paley--Zygmund\footnote{Chebyshev's inequality would give a better probability, but for our purposes it is not needed.} inequality gives the desired probability:
    $$
        \Pr[\ref{prop:Wb}] \ge \Pr[L \ge \mathbb{E}[L]/10] \ge 0.81 \cdot \frac{\mathbb{E}[L]^2}{\mathbb{E}[L^2]} > 4/5,
    $$
    for sufficiently large $m \ge T m_0$ (which can be achieved by taking large enough constant $T$).
    
    Proof of $\Pr[\ref{prop:Wc}] > 4/5$ proceeds similarly. Let $S = \sum_{v \in V(\cH)} \deg_{k'}(v, W)^2$. For a vertex $v \in V(\cH)$ and an edge $e \in \cH$ with $v \in e$, let $L_e^v$ be the indicator random variable for the event $|e \cap (W \setminus \{v\})| \ge k'$. Then
    $$
       S = \sum_{v \in \cH} \sum_{\substack{(e, e') \in \Lambda_\cH \\ v \in e \cap e'}} L_e^v L_{e'}^v.
    $$
    Note that $L_e^v L_{e'}^v = 1$ implies
    \begin{equation} \label{eq:W_x}
       |(e \cup e') \setminus W| \le 2k - 2k' - 1 =: x,
    \end{equation}
    thus
    $$
        \mathbb{E}[L_e^v L_{e'}^v] < 2^{2k} (\xi/2)^{|e \cup e'| - x}.
    $$
    Similarly, if $\mathbb{E}[L_e^v L_{e'}^v = 1] > 0$ then $|(e \cup e') \setminus I'| \le x$. Putting all together, by the linearity of expectation we have
    \begin{align*}
        \mathbb{E}[S] &= \sum_{v \in \cH} \sum_{\substack{(e, e') \in \Lambda_\cH\\v \in e \cap e'}} \mathbb{E}[L_e^v L_{e'}^v] \le \sum_{v \in \cH} \sum_{\substack{(e, e') \in \Lambda_\cH\\v \in e \cap e'}} I'_x(e \cup e') \cdot 2^{2k} (\xi/2)^{|e \cup e'| - x} \\
        &< 2k \sum_{(e, e') \in \Lambda_{\cH}} I'_x(e \cup e') \cdot 2^{2k}(\xi/2)^{|e \cup e'| - x} = 2k  2^k \cdot \Delta_{\cH}^{I'}(\xi/2, x).
    \end{align*}
    The factor $2k$ in the penultimate inequality corresponds to the fact that for each $(e, e') \in \Lambda_{\cH}$ we can choose $v \in e \cap e'$ in less than $2k$ ways. From $I' \subseteq I \setminus X$, the assumption that $I \setminus X$ satisfies the property of Lemma \ref{lemma:Lambda}, and $\xi m / 2 \ge m_0$ (which holds for $T$ sufficiently large), we further get
    \begin{align*}
        \mathbb{E}[S] &\le  2k  2^{2k} \Delta_{\cH}^{I \setminus X}(\xi/2, x) \stackrel{\eqref{eq:Delta_xi_x}}{\le} 2k 2^{2k} \cdot K \Delta_{\cH}(\xi m / 2) \cdot (\xi q / 2)^{-x} \\
        &\le 2k2^{2k} K \cdot \frac{B \mu_\cH(\xi m / 2)^2}{\xi m / 2} (\xi q / 2)^{-x} \le Z' \cdot (e(\cH) (\xi q)^{k'})^2 / N,
    \end{align*}
    where $Z' = 2k2^{2k} BK$. Finally, taking $Z = 5Z'$ in \ref{prop:Wc}, Markov's inequality gives the desired probability:
    $$
        \Pr[\ref{prop:Wc}] \ge \Pr[S \le 5\mathbb{E}[S]] \ge 4/5.
    $$
\end{proof}

The proof of Lemma \ref{lemma:Lambda} relies on \emph{deletion lemma} of R\"odl and Ruci\'nski. The lemma originally appears in \cite[Lemma 4]{rodl1995threshold} and was instrumental in a breakthrough by Schacht \cite{schacht2016extremal}. 

\begin{lemma}[Lemma A.1 in \cite{nenadov2021klr}] \label{lemma:deletion_general}
    Let $V$ be a set with $N$ elements and let $\cS$ be a family of $s$-element subsets of $V$. For every $\beta > 0$ there exists $K > 1$, such that an $m$-element subset $I \subseteq V$ $(1 \le m \le N)$, chosen uniformly at random, with probability at least $1 - \beta^m$ has the following property: There exists $X \subseteq I$ of size $|X| \le \beta m$ such that $I \setminus X$ contains at most
    $$
        K \cdot |\cS| (m / N)^{s}
    $$
    sets from $\cS$.
\end{lemma}
Note that \cite[Lemma A.1]{nenadov2021klr} requires $m < N/4$. For $m \ge N/4$ the conclusion trivially holds for $K = 4^s$, thus we may assume it holds in the whole range.

\begin{proof}[Proof of Lemma \ref{lemma:Lambda}]
Set $q = m / N$. For an integer $u \in [k, 2k-1]$, let
$$
    \Lambda_{\cH}(u) = \{ (e, e') \in \Lambda_{\cH} \colon |e \cup e'| = u \},
$$
and rewrite $\Delta_{\cH}(y)$ as
\begin{equation} \label{eq:Delta_rewrite}
    \Delta_\cH(y) = \sum_{u = k}^{2k-1} |\Lambda_{\cH}(u)| (y/N)^u.
\end{equation}
For each $u \in [k, 2k-1]$ and $s \in [u]$, define a family $\cS^s_u$ as follows:
$$
    \cS^{s}_u = \left\{ S \subseteq e \cup e' \colon (e, e') \in \Lambda_{\cH}(u) \; \text{ and } \; |S| = s \right\}.
$$
It is important to note an $s$-subset $S$ appears in $\cS_u^s$ with multiplicity $c(S) \ge 0$, where $c(S)$ counts the number of distinct pairs $(e, e') \in \Lambda_{\cH}(u)$ with $S \subseteq e \cup e'$. We make use of $|\cS_u^s| \le 2^{2k} |\Lambda_{\cH}(u)|$.

By Lemma \ref{lemma:deletion_general} (applied with $\beta/k^2$ as $\beta$) and union-bound, an $m$-element $I \subseteq V(\cH)$, chosen uniformly at random, with probability at least $1 - \beta^m$ contains a subset $X \subseteq I$ of size $|X| \le \beta m$ such that, for each $u \in [k, 2k-1]$ and $s \in [u]$, $Q = I \setminus X$ satisfies
$$
    \sum_{S \in \cS_s^u} Q(S) \le K' \cdot |\cS_s^u| q^{s}
$$
where $K' = \max_{s} K_{\ref{lemma:deletion_general}}(\beta/k^2, s)$ and $Q(S) = Q_0(S)$ denotes the indicator function for $S \subseteq Q$. For integers $u \in [k, 2k-1]$ and $x \in [0, u-1]$, we then have
$$
    \sum_{(e, e') \in \Lambda_\cH(u)} Q_x(e \cup e') \le \sum_{s = u-x}^u \sum_{S \in \cS_s^u} Q(S) \le K |\Lambda_\cH(u)| q^{u-x},
$$
where $K = 2k \cdot K' \cdot 2^{2k}$. For $x \ge u$, we use the following trivial bound:
$$
     \sum_{(e, e') \in \Lambda_\cH(u)} Q_x(e \cup e') =  \sum_{(e, e') \in \Lambda_\cH(u)} 1 \le |\Lambda_{\cH}(u)| \le |\Lambda_{\cH}(u)| q^{u-x}.
$$
Putting all together, for an integer $x \in [0, 2k-1]$ and $\xi > 0$ we have
\begin{align*}
    \Delta_\cH^Q(\xi, x) = \sum_{u=k}^{2k-1} \sum_{(e, e') \in \Lambda_{\cH}(u)} Q_x(e \cup e') \xi^{u - x} \le K \sum_{u=k}^{2k-1} |\Lambda_\cH(u)| (\xi q)^{u - x} \stackrel{\eqref{eq:Delta_rewrite}}{=} K (\xi q)^{-x} \Delta_{\cH}(\xi m).
\end{align*}
\end{proof}

% \newpage

\section{Application: Transference for removal lemmas}
\label{sec:application}

Hypergraph removal lemma states that for every $\alpha > 0$ there exists $\gamma > 0$ such that any $k$-graph with $n$ vertices and at most $\gamma n^{v(H)}$ copies of some fixed $k$-graph $H$ can be made $H$-free by removing at most $\alpha n^k$ edges. This innocent looking statement has striking implications, the most prominent being Szemer\'edi's theorem \cite{szemeredi1975sets} on arithmetic progressions in dense sets of natural numbers. For the history of the lemma, its applications and recent developments, we refer the reader to the survey by Conlon and Fox \cite{conlon13removalsurvey}. 

As an application of Theorem \ref{thm:containers} and, more importantnly, a demonstration that it is not hindered by the existence of a small number of subsets for which neither \ref{prop:i} nor \ref{prop:ii} apply, we prove the following version of a removal lemma for random hypergraphs. To state it concisely, let us denote with $\Gknp$ a subgraph of $\cK_n^{(k)}$, the complete $k$-graph with $n$ vertices, obtained by taking each hyperedge with probability $p$, independently of all other edges, for some $p \in (0, 1]$.

\begin{theorem} \label{thm:random_removal}
    For any $k$-graph $H$ and $\alpha > 0$ there exists $C, \gamma > 0$ such that if $p \ge Cn^{-1/m_k(H)}$, where
   $$
        m_k(H) = \max\left\{ \frac{e(H') - 1}{v(H') - k} \colon H' \subseteq H, v(H') > k \right\},
    $$
    then with high probability $\Gamma = \Gknp$ has the following property: Every subgraph $\Gamma' \subseteq \Gamma$ with at most $\gamma n^{v(H)} p^{e(H)}$ copies of $H$ can be made $H$-free by removing at most $\alpha n^{k} p$ edges.
\end{theorem}

It should be noted that Theorem \ref{thm:random_removal} does not have nearly as striking implications as the original hypergraph removal lemma. Nonetheless, it is an interesting statement which in some way quantifies  the distribution of the copies of $H$ in $\Gknp$. The graph case ($k=2$) of Theorem \ref{thm:random_removal} was originally obtained by Conlon, Gowers, Samotij, and Schacht \cite{conlon2015klr} using an adaption of the regularity method for random graphs. Conlon and Gowers \cite{conlon2016combinatorial} have proved the  general case under the assumption that $m_k(H)$ is obtained uniquely for $H' = H$, that is, when $H$ is \emph{strictly $k$-balanced}. Here we prove the statement without any additional assumption.

It should be noted that the proof of the removal lemma for random graphs using the regularity method very closely follows the original proof of the regularity lemma for graphs. In contrast, the transference result presented here as well as the one by Conlon and Gowers \cite{conlon2016combinatorial} uses the the hypergraph removal lemma, a notoriously difficult result, as a black box. That being said, instead of directly proving Theorem \ref{thm:random_removal} we prove a more general result which abstracts out the removal property, from which Theorem \ref{thm:random_removal} then follows as a straightforward corollary.

Given $s, r \in \mathbb{N}$, we say that a hypergraph $\cH$ is \emph{$(s, r)$-removable} if for every $I \subseteq V(\cH)$ with $e(I) < s$ there exists $X \subseteq I$ of size $|X| \le r$ such that $I \setminus X$ is an independent set in $\cH$.

\begin{theorem} \label{thm:removal_transference}
    For every $k \in \mathbb{N}$ and $B, \alpha > 0$ there exists $\gamma_0, C > 0$ such that the following holds. Let $\cH$ be a $(\gamma e(\cH), \alpha N)$-removable $k$-graph, for some $\gamma \le \gamma_0$, and suppose $m_0 \in \mathbb{N}$ is such that
    $$
        \Delta_{\cH}(m_0) \le B \mu_{\cH}(m_0)^2/m_0.
    $$
    Let $W \subseteq V(\cH)$ be a random subset formed by taking each element with probability $p \ge C m_0 / N$, independently of all other elements. Then with probability at least $1 - \exp(-\Theta(Np))$ the induced hypergraph $\cH[W]$ is $(\gamma \alpha^k e(\cH) p^k, 3 \alpha N p)$-removable.
\end{theorem}
\begin{proof}
    Let $\gamma_0 > 0$ be as given by Theorem \ref{thm:containers} for some sufficiently small $\eps = \eps(\alpha)$. Suppose $\cH$ is $(\gamma e(\cH), \alpha N)$-removable for some $0 < \gamma \le \gamma_0$. 
    
    By Chernoff's inequality, we have $Np/2 < |W| < 2 Np$ with probability $1 - \exp(-\Theta(Np))$. Moreover, a simple union-bound gives the following upper bound on the probability that an $m$-element subset excluded by Theorem \ref{thm:containers}, for $\alpha Np \le m < 2 Np$, is a subset of $W$:
    $$
        \sum_{m = \alpha Np}^{2Np} \eps^m \binom{N}{m} p^m < N \left( \frac{\eps e N}{\alpha N p} p \right)^m < e^{-N p},
    $$
    for sufficiently small $\eps = \eps(\alpha)$. Therefore, from now on we can assume that for every $m$-element subset $I \subseteq W$ with 
    \begin{equation} \label{eq:bound_I_gamma0}
        e(I) < \gamma_0 \alpha^k e(\cH)p^k \le \gamma_0 e(\cH)(m/N)^k,
    \end{equation}
    for $\alpha Np \le m < 2Np$, the property \ref{prop:ii} of  Theorem \ref{thm:containers} holds.
    
    Suppose there exists $I \subseteq W$ which satisfies \eqref{eq:bound_I_gamma0} such that $I \setminus X$ is not independent for every $X \subseteq I$ of size $|X| \le 3 \alpha Np$. In particular, this implies $|I| = m > 3 \alpha Np$. Then, by \ref{prop:ii}, $I \subseteq F \cup f_m(\mathbf{F})$ for some $F \subseteq I$ of size $|F| = \eps m$ and $\mathbf{F} \in \cT_{k-1}(F)$. As $\cH$ is $(\gamma e(\cH), \alpha N)$-removable and $e(f_m(\mathbf{F})) < \gamma e(\cH)$, there exists a subset $R(\mathbf{F}) \subset V(\cH)$ of size $|R(\mathbf{F})| \le \alpha N$ such that $f_m(\mathbf{F}) \setminus R(\mathbf{F})$ is an independent set. If $|W \cap R(\mathbf{F})| < 2 \alpha N p$ then we can make $I$ to be an independent set by removing $F \cup (W \cap R(\mathbf{F}))$ which, for $\eps < \alpha / 2$, is of size at most $3\alpha N p$. Therefore, we must have $|W \cap R(\mathbf{F})| > 2 \alpha N p$.
    
    To summarise, to show the theorem it suffices to bound the probability of an event that $|W \cap R(\mathbf{F})| \ge 2 \alpha N p$ for some $\mathbf{F} \in \cT_{k-1}(F)$ where $F \subseteq W$ is of size $|F| = s$ with $\eps \alpha Np \le s \le 2 \eps Np$. The probability of this happening for one particular $\mathbf{F}$ is, by Chernoff's inequality, at most $e^{-\alpha N p / 3}$, thus a union-bound over all possible $\mathbf{F}$ gives the following:
    $$
        \sum_{s = \alpha \eps Np}^{2 \eps Np} \binom{N}{s} 2^{ks} \cdot p^s \cdot e^{-\alpha N p / 3} \le e^{-\alpha Np/3} \sum_{s = \alpha \eps Np}^{2 \eps Np} \left( \frac{e 2^k N p}{s} \right)^s \le e^{-\alpha Np / 3} \cdot N \left( \frac{e2^k Np}{2\eps Np} \right)^{2\eps Np},
    $$
     which is of order $e^{-\Theta(Np)}$ for $\eps$ sufficiently small with respect to $\alpha$. This concludes the proof.
\end{proof}

Note that the proof works just the same using the hypergraph containers theorem of Saxton and Thomason \cite{saxton2015containers}. For comparison, a variant of Balogh, Morris, and Samotij \cite{balogh2015indset} does not suffice as it only gives $e(I) > 0$ in \ref{prop:i}. Theorem \ref{thm:random_removal} now follows as a straightforward application of Theorem \ref{thm:removal_transference} and the hypergraph removal lemma.

\begin{proof}[Proof of Theorem \ref{thm:random_removal}]
    Let $\cH$ be an $e(H)$-graph whose vertices correspond to edges in $\cK_n^{(k)}$ and $e(H)$ vertices in $\cH$ form an edge if the corresponding hyperedges in $\cK_n^{(k)}$ form a copy of $H$. By the hypergraph removal lemma, for a given $\alpha > 0$ there exists $\gamma_r > 0$ such that $\cH$ is $(\gamma_r e(\cH), \alpha N)$-removable. Furthermore, it is well known (and easy to verify) that $\Delta_{\cH}(m_0)$ satisfies the condition of Theorem \ref{thm:removal_transference} for $m_0 = n^{k - 1/m_k(H)}$ and some $B > 0$. The theorem now follows for $\gamma = \min\{\gamma_r, \gamma_0\}$, where $\gamma_0$ is as given by Theorem \ref{thm:removal_transference} for $\alpha$ and $B$.
\end{proof}

\textbf{Acknowledgment.} The author would like to thank Milo\v s Truji\'c and Andrew Thomason for comments on the early version of the manuscript. The author is also indebted to the anonymous referee for many helpful suggestions and for spotting a subtle (but serious) mistake in the setup of Theorem \ref{thm:containers}, which manifested itself in the proof of Lemma \ref{lemma:saturating}.

{\small \bibliographystyle{abbrv} \bibliography{containers}}

\end{document}